\documentclass[11pt, letterpaper]{amsart}
\usepackage{amssymb}
\usepackage[foot]{amsaddr}
\usepackage{paralist}
\newtheorem{thm}[subsection]{Theorem}
\newtheorem{prop}[subsection]{Proposition}
\newtheorem{lem}[subsection]{Lemma}
\newtheorem{rem}[subsection]{Remark}
\newtheorem{coro}[subsection]{Corollary}

\newtheorem{defin}[subsection]{Definition}

\newcommand{\ds}{\displaystyle}
\newcommand{\eps}{\varepsilon}
\newcommand{\Nat}{{\bf N}}
\newcommand{\Prob}{{\bf P}}
\newcommand{\E}{{\bf E}}
\newcommand{\EE}{{\mathcal E}}

\newcommand{\bpm}{\begin{pmatrix}}
\newcommand{\epm}{\end{pmatrix}}

\begin{document}

\title[A P.A. Process Approaching the Rado Graph]{A Preferential Attachment Process Approaching the Rado Graph}
\author{Richard Elwes}
\address{School of Mathematics, University of Leeds, UK}
\email{r.h.elwes@leeds.ac.uk}

\begin{abstract}

We consider a simple Preferential Attachment graph process, which begins with a finite graph, and in which a new $(t+1)$st vertex is added at each subsequent time step $t$, and connected to each previous vertex $u \leq t$ with probability $\frac{d_u(t)}{t}$ where $d_u(t)$ is the degree of $u$ at time $t$. We analyse the graph obtained as the infinite limit of this process, and show that so long as the initial finite graph is neither edgeless nor complete, with probability 1 the outcome will be a copy of the Rado graph augmented with a finite number of either isolated or universal vertices. 

\end{abstract}

\maketitle

\section{Introduction}

Since its introduction by Barab\'asi and Albert in \cite{BA}, the mechanism of Preferential Attachment (PA) has been highly influential amongst scientists seeking to model real-world networks. In PA processes, a new vertex is introduced at each time step, and then connected to each pre-existing vertex with a probability depending on the current degree of that vertex. The study of PA and related processes thus presents a new challenge in the mathematics of random graphs, differing from the classical approach of the Erd\H{o}s-R\'enyi school who principally study structures arising from the following simpler process: at each time step introduce a new vertex, and connect it to each previous vertex with some fixed probability $p$.

Thus, an important question is whether the mathematics of PA processes can be developed to the same advanced level as the Erd\H{o}s-R\'enyi theory. Recall that this theory has two distinct facets. First, researchers have analysed in great detail the finite graphs which emerge. Here, questions of interest include the emergence of a giant component and the degree distribution of the vertices, and analyses are typically highly sensitive to the value of $p$. See \cite{Bol} for a comprehensive discussion of such matters.

The second angle of approach is to understand the infinite limit of the process. In this case, a remarkable theorem of Erd\H{o}s and R\'enyi guarantees that, irrespective of the value of $p \in (0,1)$, the resulting graph will with probability $1$ be isomorphic to the following important structure:

\begin{defin} \label{defin:Rado}
\emph{The Rado graph} is a graph on a countably infinite set of vertices satisfying the following: given any disjoint finite sets of vertices $U$ and $V$, there exists a vertex $v$ connected to each vertex in $V$ and none in $U$. 
\end{defin}

This graph exhibits many interesting properties. To start with (and justifying our use of the definite article) it is \emph{countably categorical}, meaning that any two graphs obeying the definition are automatically isomorphic. Beyond this, all finite and countably infinite graphs isomorphically embed in it. For these and other reasons, this structure is an object of central importance in several branches of mathematics. See \cite{C} for a recent survey.

This division between the finite and infinite is equally applicable to the study of PA processes. On the finite side, a good amount of progress has been made. Notably, in \cite{DM} and \cite{DM2}, Dereich and M\"{o}rters analyse the following family of Preferential Attachment Processes: at stage $t$ we have a directed graph $\textrm{DG}(t)$ into which a new vertex $t+1$ is introduced. For each previous vertex $u$, an edge from $t+1$ to $u$ is introduced independently with probability $\frac{f(I_u(t))}{t}$ where $f$ is a fixed sublinear function, and $I_u(t)$ is the indegree of $u$ in $\textrm{DG}(t)$. The model considered in this paper (Definition \ref{defin:ourprocess} below) is essentially a linear version of theirs, although we consider our structures as \emph{undirected} graphs. Dereich and M\"{o}rters successfully extract a great deal of valuable information about this process, including the distributions of in- and outdegrees.

This follows earlier work from Bollob\'as, Riordan, Spencer, and Tusn\'ady (\cite{Bol2}) and  Bollob\'as and Riordan (\cite{Bol3}) analysing networks arising from a preferential attachment process in which a fixed number $m$ of edges are added at each stage, and connected to previous vertices with probability directly proportional to their degrees.

The infinite limits of such processes have received less analysis, and this is our focus here. Our entry point is the paper of Kleinberg and Kleinberg \cite{KK}, in which a process is studied whereby a single vertex and a constant number $C$ of edges are added at each time-step, with each new edge starting at the new vertex and with endpoint independently chosen among the pre-existing vertices with probability proportional to their degree. Thus, these structures are analysed as \emph{directed multi-graphs}, in that each edge has a direction, and there may exist two or more edges sharing the same start and end-points. (Loops from a vertex to itself are not permitted, however.)

Kleinberg and Kleinberg show that in each of the cases $C=1$ and $C=2$, there is, up to isomorphism, a unique infinite limiting structure, which the process approaches with probability $1$. (They also show that the analogous result fails for $C \geq 3$.) 

We proceed in a similar spirit, pinning down the limiting structure up to isomorphism. However, in our model, the number of edges added at each stage is not prescribed, but is itself a random variable. Specifically, our model connects the new vertex $t+1$ to each previous vertex $u \leq t$ with probability $\frac{d_u(t)}{t}$, where $d_u(t)$ is the degree of $u$ at time $t$. (As mentioned above, in the current model, edges are directionless, and parallel edges are not permitted.)

Thus, our graphs are denser than those of Kleinberg and Kleinberg; where the expected number of edges in their graphs is linear in the number of vertices, in our case it is quadratic (this is made explicit in Lemma \ref{lem:anothermart} below).

In Theorem \ref{thm:main} below, we establish that, so long as the initial graph is neither complete nor edgeless, our process will with probability 1 approach the Rado graph, or a modification of it in which a finite number of universal or isolated vertices are incorporated. In Section \ref{section:close}, we outline an adaptation of the foregoing machinery to a variation of the model with probability of attachment given by $\lambda \cdot \frac{d_u(t)}{t}$ for some constant $\lambda \in (0,1]$. We make the case (omitting the delicate conditional probability considerations) that if $\lambda <1$, with probability 1 the limiting structure will be non-isomorphic to an augmented Rado graph. 

In \cite{E1} (as yet unpublished work developed simultaneously with the current paper) the author analyses a family of PA processes in which parallel edges are permitted, and in which the number of edges added at each stage is prescribed. Specifically it is shown that if $f(t)$ edges are added at time $t$, where $f$ is asymptotically bounded above and below by linear functions in $t$, then with probability $1$ the process will approach the natural multigraph analogue of the Rado graph. (We refer the reader to Definition 2.1 of \cite{E1} for the formal definition.)

\section{The Process}

\begin{defin} \label{defin:ourprocess}
Let $G'$ be any finite graph containing at least two vertices. We take its vertex set to be $\{0,1,\ldots, v'\}$, and let $G(t):=G'$ for all $t \leq v'$. For $t \geq v'$, we create a new graph $G(t+1)$ by introducing a new vertex $t+1$ which is connected to each previous vertex $u \leq t$ with probability $\ds p_u(t+1)=\frac{d_u(t)}{t}$ where $d_u(t)$ is the degree of vertex $u$ in $G(t)$.
\end{defin}

Notice that in the current model, as in \cite{DM}, but in contrast to \cite{E1} and \cite{OS}, the number of edges added at each stage is itself a random variable.

\begin{rem} \label{rem:nonstan}
$G(t)$ contains $t+1$ vertices, and thus $d_u(t)$ may take any value between $0$ and $t$. Now if $d_u(t')=0$ for any $t'$ then automatically $d_u(t)=0$ for all $t \geq u$. Likewise if $d_u(t')=t'$ then $d_u(t)=t$ for all $t \geq u$. Thus isolated vertices remain isolated, and universal vertices (i.e vertices connected to every other) remain universal. Of course, no graph can contain both.

We shall call isolated and universal vertices \textbf{non-standard} and all others \textbf{standard}.

\end{rem}

Our interest is the infinite limit of this process $G(\infty)$, and we shall prove:

\begin{thm} \label{thm:main}
\begin{enumerate}
\item If $G'$ is complete, then $G(\infty)$ is complete.

\item If $G'$ is edgeless, then so too is $G(\infty)$.

\item For any other $G'$, with probability $1$, the infinite limit $G(\infty)$ is isomorphic to one of the following:
\begin{itemize}
\item The Rado graph, augmented with a finite number of isolated vertices.
\item The Rado graph, augmented with a finite number of universal vertices.
\end{itemize}
\end{enumerate}
\end{thm}

The first two clauses of Theorem \ref{thm:main} are immediate from Remark \ref{rem:nonstan} above, so we concentrate on the third, and make the standing assumption that $G'$ is neither complete nor edgeless (from which it immediately follows that no stage is complete or edgeless).

Since, by Remark \ref{rem:nonstan}, non-standard vertices can be recognised as soon as they appear and are of little interest, we might amend the process by immediately discarding each non-standard vertex when it appears. Theorem \ref{thm:main} then guarantees that we will obtain the Rado graph as the infinite limit of this modified process.

It will be convenient to amend the model as follows: we colour each edge as described in Definition \ref{defin:ourprocess} black, and introduce the new rule that every pair of distinct vertices in $G(t)$ is connected by a white edge if and only if it is not connected by a black edge. (Thus the white graph is the complement of the black graph. We imagine that on white paper, these edges will become invisible.) Now let $d^w_u(t)$ be the white degree of the vertex $u$ at time $t$, and $d^b_u(t)$ be its black degree. This introduces a useful symmetry to the process:

\begin{rem} \label{rem:sym}
At any stage $(t)$, for any vertex $u \leq t$, it holds that $d^w_u(t) + d^b_u(t) = t$. Thus, in $G(t+1)$, the probability that vertex $u$ is connected to vertex $t+1$ with a white edge is precisely $1 - \frac{d^b_u(t)}{t} = \frac{d^w_u(t)}{t}$.\\ 
\end{rem}

\section{The Standard Vertices}

\begin{rem}
For a fixed vertex $u$, we can consider the edges of each colour emanating from $u$ as black or white balls within a P\'olya urn. Our process then is identical to selecting, at each stage, a ball from the urn uniformly at random, and then returning it along with a second ball of the same colour. This observation would allow us to appeal to known facts about the limiting proportions of black balls in the urn as given by a beta distribution $\textrm{Beta}(d_u^{b}(t_0),d_u^{w}(t_0))$. See, for instance, \cite{Freed}. From this we could easily derive the main result of this section: Proposition \ref{prop:poslimit}.

Nevertheless, we opt to provide a self-contained elementary argument, since we shall not need the details of the beta distribution, and since also our argument's main technical ingredient, Proposition \ref{prop:generalmart}, will in any case be required for separate purposes in Section \ref{section:nonstandard}.
\end{rem}

\begin{lem} \label{lem:nonzero}
For any standard vertex $u$ in any stage $G(t_0)$, the probability that $u$ never receives another black edge is $0$. The same goes for white edges.
\end{lem}

\begin{proof}

Suppose at time $t_0$ that $u$ has black degree $d(t_0)=D>0$. The probability that $u$ never receives another black edge is therefore 
$$\prod_{t=t_0}^{\infty} \left(1 - \frac{D}{t} \right).$$

We shall show that this is $0$. Taking logarithms, it is therefore enough to show that
$$\sum_{t=t_0}^{\infty} \ln \left(1 + \frac{D}{t-D}  \right)$$ diverges to $\infty$.
This follows from the divergence of the harmonic series, since for all small enough $x$, we know $\ln(1+x)> \frac{1}{2}x$.

An identical argument applies, replacing `black' with `white'.
\end{proof}

\begin{coro} \label{coro:dinf}
Given any state of the graph $G(t_0)$ containing any standard vertex $u$, with probability $1$ it will be true that $d_u^b(t) \to \infty$ and $d_u^w(t) \to \infty$ as $t \to \infty$.
\end{coro}

\begin{proof}
This follows automatically from Lemma \ref{lem:nonzero} by the countable additivity of the probability measure.
\end{proof}

\begin{defin}
Given a vertex $u$ and a colour $c \in \{b,w\}$, define the following:
\begin{align*}
U^c_u(t+1)&:=d^c_u(t+1) - d^c_u(t)\\
X^c_u(t)&:=\frac{d^c_u(t)}{t}
\end{align*}
\end{defin}

We shall suppress the subscript $u$ and superscript $c$ in the above when they are clear from context.

\begin{lem} \label{lem:xuexists}
For any vertex $u$ and colour $c \in \{b,w\}$
\begin{enumerate}[(i)]
\item For all $t$, $X_u^c(t) \in [0,1]$. If $u$ is standard then $X_u^c(t) \in (0,1)$. 
\item $\left( X_u^c(t) \right)_{t \geq u}$ is a martingale.
\item With probability 1, there exists $x^c_u \in [0,1]$ such that $\frac{d^c_u(t)}{t} \to x^c_u$.
\item Furthermore, $x^b_u + x^w_u =1$.
\end{enumerate}
\end{lem}

\begin{proof}
The first part is immediate from the definitions. 
Writing $X$ for $X_u^c$, we show that $X(t)$ is a martingale:
\begin{align*}
\E \left( X(t+1) || X(t) = \frac{d}{t} \right) &= \frac{1}{t+1} \E (d(t+1))||d(t) = d)\\
&= \frac{1}{t+1}\left(d+ \frac{d}{t} \right)  =  \frac{d}{t} = X(t).\end{align*}

Part (iii) follows from parts (i) and (ii) by Doob's convergence theorem. Part (iv) follows too, since $X^b(t)+X^w(t) = 1$ by Remark \ref{rem:sym}. \end{proof}

\begin{rem} \label{rem:L2}
By Lemma \ref{lem:xuexists} (i), the martingale  $\left( X^c_u(t) \right)_{t \geq u}$ is bounded in $\mathcal{L}_2$.
\end{rem}

Our next goal is to show (in Proposition \ref{prop:poslimit} below) that with probability $1$ we have $0<x^c_u<1$. Towards this, we recall some Martingale machinery from \cite{KK}, which we express more generally for subsequent reuse. First recall the Kolmogorov-Doob inequality: 

\begin{thm}[Kolmogorov-Doob Inequality] \label{thm:KD}
Suppose that $\big( \tilde{Z}(n) \big)_{\! n \in \Nat }$ is a non-negative submartingale and $\alpha>0$. Then for any $N  \in \Nat$
$$\Prob \left( \max_{n \leq N} \tilde{Z}(n) \geq \alpha \right) \leq \frac{1}{\alpha} \E \left( \tilde{Z}(N) \right).$$
\end{thm}

\begin{prop} \label{prop:generalmart}
Suppose that $\big( Z(t) \big)_{\! t  \in \Nat}$ is a Martingale with limit $z$, such that with probability $1$ there exist $\alpha, A >0$ and $t_1 \geq 3$ (each of whose values may be given by random variables) so that:
\begin{enumerate}[(i)]
\item $Z(t) \in (0,1)$ for all $t \geq t_1$
\item $t \cdot Z(t) > \alpha$ for all $t \geq t_1$.
\item For all $t>m \geq t_1$ 
$$\E \Big( Z(t+1)^2 - Z(t)^2 \Big| \Big| Z(m) \Big) < \frac{A}{t^2} \cdot Z(m).$$
\item $\ds \beta := \frac{8A}{\alpha} <1$.
\end{enumerate}

\noindent Then $\Prob(z>0)=1$.
\end{prop}

\begin{proof}

The argument is essentially contained in \cite{KK}, but we include it here for completeness. Given any $n>m \geq t_1$ define $$\tilde{Z}_{m}(n) := \left( Z(n) - Z(m) \right)^2.$$ 
Thus by our third hypothesis, for all $n>m$,
\begin{align*}
\E \left( \tilde{Z}_{m}(n) \big|\big| Z(m) \right) & = \sum_{t=m}^{n-1} \E \Big( Z(t+1)^2 \big|\big| Z(m) \Big) - \E \Big( Z(t)^2 \big|\big| Z(m) \Big)\\
& < A \cdot Z(m) \cdot \sum_{t=m}^{\infty} \frac{1}{t^2} \ \ < \ \frac{2A}{m} \cdot Z(m)
\end{align*}
since $m \geq 3$.

Notice too that for fixed $m$, the sequence $\left( \tilde{Z}_{m}(n) \right)_{\! n}$ forms a submartingale.

Now beginning at any time $t_0$, we may define a sequence of times: $n_0 = t_0$. Let $n_{i+1}$ be the least $n \geq n_i$ (if any exists) such that $Z(n) < \frac{1}{2}Z \left(n_{i} \right)$. Otherwise $n_{i+1} = \infty$.

We next apply the Kolmogorov-Doob inequality (Theorem \ref{thm:KD}) to $\tilde{Z}_{n_i}(n)$:
\begin{eqnarray*}
\Prob ( n_{i+1} < \infty ||n_{i} < \infty) & = & \Prob \left(  \min_{n \geq n_i} Z(n) < \frac{1}{2} Z(n_i) {\Big|\Big|} Z(n_i)\right)\\
& \leq & \Prob \left(  \max_{n \geq n_i} \tilde{Z}_{n_i}(n) > \frac{1}{4} Z(n_i)^2 {\Big|\Big|} Z(n_i)\right) \\
& = & \lim_{N\to \infty} \Prob \left(  \max_{n: N \geq n \geq n_i} \tilde{Z}_{n_i}(n) > \frac{1}{4} Z(n_i)^2 {\Big|\Big|} Z(n_i)\right) \\
& \leq & \frac{4}{Z(n_i)^2} \cdot \lim_{N \to \infty} \E (\tilde{Z}_{n_i}(N) || Z(n_i)).\end{eqnarray*}

The above holds if we start the at any time $t_0$. However, with probability 1, the time $t_1$ exists as described. In the case $t_0 \geq t_1$, we find $$\Prob ( n_{i+1} < \infty ||n_{i} < \infty) < \frac{8A}{Z(n_i) \cdot n_i} \ \ < \beta.$$

Since $\beta<1$, it follows that the event that every $n_i$ is finite has probability zero, giving the result. \end{proof}

\begin{prop} \label{prop:poslimit}
Given any standard vertex $u$, we have $\Prob \left(x^c_u>0 \right) =  \Prob \left(x^c_u<1 \right) = 1$.
\end{prop}

\begin{proof}
As usual we write $X$ for $X^c_u$ and $x$ for $x^c_u$.

The result will follow by Proposition \ref{prop:generalmart} applied to $X(t)$, once we have verified its hypotheses. The first follows by the assumption that $u$ is standard. For second, note that by Corollary \ref{coro:dinf} with probability 1 we will see $d^c_u(t_0) \geq 16$ (say) for some $t_0$. Thus we may take $\alpha = 16$.

For the third, recall $U(t+1) := d(t+1) - d(t)$. So $U(t+1)$ is a Bernoulli variable with $$\E \Big( U(t+1) \Big| \Big| d(t) = d \Big) = \E \Big( U(t+1)^2 \Big| \Big| d(t) = d \Big) = \frac{d}{t}.$$

Since $d(t+1)^2 = d(t)^2+2d(t) U(t+1) + U(t+1)^2$,
\begin{align*}
\E \Big( d(t+1)^2 || d(t)=  d \Big) & = d^2+ 2d \cdot \frac{d}{t} + \frac{d}{t}\\
&< \frac{d}{t} + d^2\left(1 + \frac{1}{t} \right)^2\\
&= X + X^2\left(t+1 \right)^2.
\end{align*}
Now
\begin{eqnarray*}
\E \Big( X(t+1)^2 & \ \Big| \Big| & d(t)=d \Big) =\frac{1}{(t+1)^2} \cdot \E \Big( d(t+1)^2\ \Big| \Big | \ d(t)=d \Big)\\
& < & \frac{1}{(t+1)^2} X(t) + X(t)^2.
\end{eqnarray*}
Since $\E \Big( X(t) \Big| \Big| X(m) \Big) = X(m)$, by the law of total expectation,
\begin{equation}
\E \Big( X(t+1)^2 - X(t)^2 \Big| \Big| X(m) \Big) < \frac{1}{t^2} X(m). \label{equation:squares} \end{equation}
Taking $A=1$ gives the result.
\end{proof}

We may now prove the first part of Theorem \ref{thm:main}.

\begin{prop} \label{prop:proof1}
The subgraph of $G(\infty)$ comprising the standard vertices is isomorphic to the Rado graph.
\end{prop}

\begin{proof}
Let $U= \left\{u_i : i \leq n \right\}$ and  $V= \left\{v_j : j \leq m\right\}$ be disjoint finite sets of standard vertices, and let $t_0$ be any stage of the process. Our initial goal is to establish the existence of a \emph{witness} for $(U,V)$, i.e. a vertex adjacent to each $u_i$ and no $v_j$, which emerges at some time $t_1 > t_0$.

For each $u_i$ (respectively $v_j$) by Lemma \ref{lem:xuexists} and Proposition \ref{prop:poslimit} there exists $x_i$ (respectively $y_j$) in  $(0,1)$ representing the limiting proportion of neighbours of $u_i$ ($v_j$) among all vertices. Since a vertex $w> \max \left\{u_i,v_j : i \leq n, j \leq m \right\}$ is connected to each $u_i$ or $v_j$ independently of whether it is connected to the others, it follows with probability $1$ that the limiting proportion of witnesses for $(U,V)$ is given by $\prod_{i=1}^n x_i \cdot \prod_{j=1}^m \left( 1-y_j \right)>0$. Thus with probability $1$, there are infinitely many such witnesses, and thus at least one arising after stage $t_0$.

To complete the proof, we enumerate all possible pairs of disjoint finite sets $\left(U_i, V_i\right)_{i \in \Nat}$. Then let $t_0$ be the first stage at which $\left(U_0, V_0\right)$ receives a witness and $t_{i+1}$ be the first stage strictly after $t_i$ at which $\left(U_{i+1}, V_{i+1}\right)$ receives a witness.

Then by the argument above and the countable additivity of the probability measure, with probability 1 every $t_i$ is finite as required. \end{proof}

\section{The non-standard vertices} \label{section:nonstandard}
Recall our standing assumption that $G'$ (and thus every stage of the process) is neither complete nor edgeless. Thus we remain under the hypotheses of part 3 of Theorem \ref{thm:main}. Our goal in this section is Proposition \ref{prop:nonstand}, where we show that the non-standard vertices in $G(\infty)$ remain finite with probability 1. As in the previous section, we employ the machinery of Martingales.

\begin{lem} \label{lem:anothermart}
Write $\EE^c(t)$ for the number of edges of colour $c$ in $G(t)$. 

Define $Y^c(t):= \frac{\EE^c (t)}{t(t+1)}$. Then 
\begin{enumerate}[(i)]
\item $Y^c(t) \in (0,1)$
\item $\left(Y^c(t)\right)_{t}$ is a Martingale
\item $Y^c(t) \to y^c$ for some $y^c \in [0,1]$. 
\end{enumerate}
\end{lem}

\begin{proof}
Omitting the superscript $c$, since $G(t)$ is not edgeless (in either colour) we know $Y(t)>0$. Also, $\ds \EE(t)< \left( \! \begin{array}{c} t \\ 2 \end{array} \! \right)$ implying that $Y(t)<1$. Now,  $$\E \left( d_{t+1}(t+1) \Big| \Big| G(t) \right)  = \sum_{u=0}^t \frac{d_u(t)}{t}= \frac{2 \EE(t)}{t}.$$ 
Also $\EE(t+1) = \EE(t)+d_{t+1}(t+1)$. It follows that 
$$\E \left( \EE(t+1)  \big| \big| G(t) \right) = \left( \frac{t+2}{t} \right) \cdot \EE(t) .$$
Thus $\E \Big( Y(t+1)) \big| \big| Y(t) \Big) = Y(t)$, and thus $Y$ is a Martingale. Part (iii) follows by Doob's Convergence Theorem.
\end{proof}

\begin{rem} \label{rem:L22}
The Martingale $Y(t)$ converges in $\mathcal{L}_2$.
\end{rem}

Our next aim is to show that with probability $1$ in fact $y >0$. This will again follow from Proposition \ref{prop:generalmart}, once we have shown that its hypotheses hold. The first is part (i) in Lemma  \ref{lem:anothermart}. We turn our attention to the second hypothesis:

\begin{prop} \label{prop:lotsofedges}
Given any state of the graph $G(t_0)$ and $c \in \{b,w\}$ with probability $1$ there exist $\alpha >0$ and $t_1 \geq t_0$ such that for all $t \geq t_1$  $$t \cdot Y^c(t) > \alpha.$$
\end{prop}

\begin{proof}
We suppress the superscript $c$ as usual. Notice first that by definition $t \cdot Y(t) =  \frac{1}{t+1} \cdot \EE(t)$.

Let $\eps>0$. We show that suitable $\alpha$ and $t_1$ exist with probability $>1-\eps$. Pick a standard vertex $u$ in $G(t_0)$ and pick $i$ large enough that $2^{-i} < \eps$. Then by Proposition \ref{prop:poslimit}, with probability $>1 - \eps$ we have $x_u \geq \frac{1}{2^i} \cdot \frac{8}{t_0}$. Thus for some large enough $t_1$ and all $t \geq t_1$ we have $d_u(t) > \frac{1}{2^i} \frac{4}{t_0} \cdot t$ and thus $\EE(t) > \frac{1}{2^i} \frac{4}{t_0} \cdot t$ and $\frac{1}{t+1}\EE(t) > \alpha:= \frac{1}{2^i} \frac{2}{t_0}$. 
\end{proof}

\begin{prop} \label{prop:ybound}
$$\E \left( Y(t+1)^2 \big| \big| Y(t) \right) - Y(t)^2 <  \frac{2}{t^3} \cdot Y(t).$$
\end{prop}

\begin{proof}
Suppressing the superscript $C$, and writing $\EE$ for $\EE(t)$ and $d$ for $d_{t+1}(t+1)$, we have
\begin{align*}
 & \ \E \left( Y(t+1)^2 \big| \big| Y(t) \right) - Y(t)^2 \\
  = & \  \frac{1}{(t+1)^2(t+2)^2}\E \left( \left( \EE + d \right)^2 \big| \big| Y(t) \right) - \frac{1}{t^2(t+1)^2}\EE^2\\
   = & \  \frac{1}{ t^2(t+1)^2(t+2)^2} \Big( t^2 \left( \EE^2 + 2 \EE \E \left( d \big| \big| Y(t) \right) \right)  + \E \left( d^2 \big| \big| Y(t) \right) - (t+2)^2 \EE^2 \Big).  
\end{align*}

Now, conditioning on $\EE(t)$, the random variable $d=d_{t+1}(t+1)$ has a Poisson-Binomial distribution with expectation $\frac{2 \EE (t)}{t}$ and variance $$\sum_{u=0}^t \left( 1 - \frac{d_u(t)}{t} \right) \frac{d_u(t)}{t} <
\sum_{u=0}^t \frac{d_u(t)}{t} = \frac{2 \EE (t)}{t}\\$$
meaning that \begin{align*}
\E \left( d^2 \big| \big| Y(t) \right) & =  \E \left( d \big| \big| Y(t) \right)^2 + \textbf{Var} \left( d \big| \big| Y(t) \right)\\
& < \frac{4 \EE^2}{t^2} + \frac{2 \EE}{t} .
\end{align*}

Thus 
\begin{align*}
 & \E \left( Y(t+1)^2 \big| \big| Y(t) \right) - Y(t)^2 \\
 & \ \ \ <  \ \  \frac{1}{t^2(t+1)^2(t+2)^2} \Big( t^2 \left( \EE^2 + \frac{4 \EE^2}{t} + \frac{4 \EE^2}{t^2} + \frac{2 \EE} {t}  \right) - (t+2)^2 \EE^2 \Big)\\
 & \ \ \ =   \ \   \frac{2 t \EE}{t^2(t+1)^2(t+2)^2} \\
 & \ \ \ =  \ \  \frac{2}{(t+1)(t+2)^2} \cdot Y(t) \\
& \ \ \ <  \ \ \frac{2}{t^3} \cdot Y(t). \end{align*} \end{proof}

\begin{prop} \label{prop:anothermartpos}
Given any state of the process $G(t_0)$ and $c \in \{b,w\}$, with probability $1$ we have $y^c > 0$. 
\end{prop}

\begin{proof}
As discussed above, we proceed using Proposition \ref{prop:generalmart}. The first hypothesis comes from Lemma \ref{lem:anothermart} and the second from Proposition \ref{prop:lotsofedges}. The third follows from Proposition \ref{prop:ybound} above:
we may pick $A$ as we like, say $A=\frac{\alpha}{16}$. By increasing $t_1$ by some predictable amount, we obtain that for $t \geq t_1$ we have 
\begin{equation} \label{equation:squaresagain}
\E \left( Y(t+1)^2  - Y(t)^2 \big| \big| Y(t) \right) < \frac{A}{t^2} \cdot Y(t).
\end{equation}
The law of total expectation now gives the third hypothesis, and the fourth is immediate from our choice of $A$, giving the result.
\end{proof}

The following result completes the proof of Theorem \ref{thm:main}:

\begin{prop} \label{prop:nonstand}
Given any state $G(t_0)$, the number of non-standard vertices in $G(\infty)$ will be finite with probability 1.
\end{prop}

\begin{proof} 

Let $\eps >0$ and let ${\bf \Delta}$ be the event $\big| \{ t \ : \ d_{t}(t) = 0 \} \big| < \infty$. Our goal is to show that $\Prob \left( {\bf \Delta } \right) \geq 1 - \eps.$ (This and everything that follows is conditioned upon $G(t_0)$, which we suppress.)

First, by Proposition \ref{prop:anothermartpos}, there exists $\gamma>0$ such that \begin{equation} \label{equation:nonstbound1}
\Prob(y<2\gamma)< \eps.\end{equation} (The 2 is included simply for our convenience.) 
Next, by Proposition \ref{prop:ybound} above and the law of total expectation, for any $t \geq t_1 \geq t_0$, we have $$\E \left( Y(t+1)^2 - Y(t)^2 \Big| \Big| Y(t_1) \right) < \frac{2}{t^3}.$$
Summing successive expressions, by the linearity of expectation, there is some $C>0$, so that for all $T \geq t$ we see
$$\E \left( Y(T+1)^2 - Y(t)^2 \Big| \Big| Y(t_1) \right) < \frac{C}{t^2}.$$
Thus, by linearity and Remark \ref{rem:L22}, 
\begin{equation}
\E \left( y^2 - Y(t)^2 \Big| \Big| Y(t_1) \right) < \frac{C}{t^2}. \label{equation:ybound2} \end{equation}
Furthermore, by Remark \ref{rem:L22} again, $\E \left( y  \ \big| \big| \ Y(t) \right) = Y(t)$ and thus 
$$\E \left( \left( y - Y(t) \right)^2 \ \big| \big| \ Y(t) \right) = \E \left( y^2 - Y(t)^2 \ \big| \big| \ Y(t) \right)$$
and so by the law of total expectation and Bound (\ref{equation:ybound2})
$$\E \left( \left( y - Y(t) \right)^2 \right)  < \frac{C}{t^2}.$$
Thus by Markov's Theorem, for any $\delta>0$, 
$$\Prob \left( | y - Y(t)|> \delta   \right)< \frac{C}{\delta^2 t^2}.$$ Since $\sum_{t=t_0}^\infty \frac{C}{\delta^2 t^2} < \infty$, by the Borel-Cantelli Lemma, the number of $t$ for which it holds that $| y - Y(t)|> \delta$ will be finite with probability 1.

Taking $\delta = \gamma$ and noticing that $\EE (t)<\gamma t^2  \Rightarrow Y(t)< \gamma$, define
${\bf \Gamma}$ to be the event $\big| \{ t \ : \ \EE (t)<\gamma t^2 \  \} \big| < \infty.$
Then \begin{equation} \label{equation:nonstbound2} \Prob \left( {\bf \Gamma} \; \big| \big| \; y \geq 2 \gamma \right) = 1. \end{equation}

Also, for any $\xi>0$ (the case $\xi = \gamma$ is of primary interest),
\begin{align*}
\E \Big( d_{t+1}(t+1) \Big| \Big| \EE(t) \geq \xi t^2 \Big) & =\E \left( \sum_{u=0}^t U_u(t) \Big| \Big| \EE(t) \geq \xi t^2 \right)\\ 
& = \E \left( \sum_{u=0}^t \frac{d_u(t)}{t} \Big| \Big| \EE(t) \geq \xi t^2 \right)\\ 
& = \frac{2 \cdot \E \left( \EE(t) \big| \big| \EE(t) \geq \xi t^2 \right) }{t} \geq 2 \xi t. \end{align*}

Since each $U_u(t) \in [0,1]$, applying Hoeffding's inequality gives
\begin{align*}
& \ \Prob \left( d_{t+1}(t+1) = 0 \ \Big| \Big| \ \EE(t) \geq \xi t^2 \right) \\
\leq & \ \Prob \Big( \E\left(d_{t+1}(t+1) \right) - d_{t+1}(t+1)  > \xi t  \ \Big| \Big| \  \EE(t) \geq \xi t^2 \Big) \\ 
 \leq & \ e^{-\frac{2 \xi^2 t^2}{t+1}} < e^{-\xi^2 t}.
\end{align*}

Furthermore, clearly $\E \left(  y \ \big| \big| \ \EE(t) \geq  \gamma t^2 \right) \geq \E \left(  y \right)$, and thus $\Prob \left(  y \geq 2 \gamma \ \big| \big| \ \EE(t) \geq  \gamma t^2 \right) \geq \Prob \left(  y \geq 2 \gamma \right)$. Hence by Bayes' Theorem, 
\begin{align*}
\Prob \left(\EE(t) \geq  \gamma t^2 \ \Big| \Big| \ y \geq 2 \gamma \right) & = \frac{\Prob \left(  y \geq 2 \gamma \ \big| \big| \ \EE(t) \geq  \gamma t^2 \right)}{\Prob \left(  y \geq 2 \gamma \right)} \cdot \Prob \left(\EE(t) \geq  \gamma t^2 \right) \\
& \geq \Prob \left( \EE(t) \geq \gamma t^2 \right).
\end{align*}
By (\ref{equation:nonstbound2}) $$\Prob \left( \EE(t) \geq \gamma t^2 \ \Big| \Big| \ {\bf \Gamma} \; \& \; y \geq 2 \gamma \right) = \Prob \left(\EE(t) \geq  \gamma t^2 \ \Big| \Big| \ y \geq 2 \gamma \right)$$ so 
$$ \Prob \left( \EE(t) \geq \gamma t^2 \ \Big| \Big| \ {\bf \Gamma} \; \& \; y \geq 2 \gamma \right) \geq \Prob \left( \EE(t) \geq \gamma t^2 \right).$$ 
Likewise
\begin{align*}\Prob \Big( d_{t+1}(t+1) = 0 \ \Big| \Big| \ \EE(t) \geq \gamma t^2 \ & \& \ {\bf \Gamma}  \; \& \; y \geq 2 \gamma \Big) \\ & \leq \Prob \left( d_{t+1}(t+1) = 0 \ \Big| \Big| \ \EE(t) \geq \gamma t^2 \right) \\ 
& < e^{-\gamma^2 t}.
\end{align*}
Now $\sum_{t=t_0}^{\infty} e^{-\gamma^2 t}$ converges to a finite limit, and thus by the Borel-Cantelli Lemma again, we see that \begin{equation} \label{equation:nonstbound3} \Prob \left({\bf \Delta} \ \big| \big| \  {\bf \Gamma} \;\& \; y \geq 2 \gamma \right) =1. \end{equation} The result follows from (\ref{equation:nonstbound1}), (\ref{equation:nonstbound2}), and (\ref{equation:nonstbound3}). 
\end{proof}

\section{Closing Comments on Generalisations} \label{section:close}

The current work considers only one model (modulo the initial graph), so it is natural to ask whether the result generalises to related models, such as one with attachment probability $\lambda \cdot \frac{d}{t}$ for some $\lambda \in (0,1]$. We offer an informal argument that it is only in the case $\lambda = 1$ that the infinite limit will be the Rado graph.

Much of the preceding theory goes through. In particular, one can define a Martingale: 
$$X_u(t):=\frac{d(t)}{u \cdot \prod_{j=u}^{t-1} \left( 1 + \frac{\lambda}{j} \right)}.$$
Considering the limit of this Martingale and the characterisation of the $\Gamma$-function as $\Gamma(\lambda):= \lim_{t \to \infty} \frac{t^\lambda}{\prod_{j=0}^t (1+\frac{\lambda}{j})}$ one may find $C \geq 0$ so that for all large enough $t$, we have $d(t) \leq C \cdot t^\lambda.$

Now consider (as per the proof of Proposition \ref{prop:proof1}) a witness-request $(U,V)$, where $|U| = n$. (For simplicity, we can take $V= \emptyset$.) For all large enough $t$, the probability of a suitable witness appearing at time $t$ is bounded above by $c \cdot t^{(\lambda -1) n}$ for some $c>0$. Thus, omitting the intricacies of various conditional probability calculations, it is only in the case $\lambda  =1$ that we expect to be able to guarantee the eventual appearance of such a witness. In particular, among the infinitely many witness requests $(U,V)$ where $|U|> \frac{1}{1 - \lambda}$, with probability 1 at least one will fail to be satisfied, making the limit non-isomorphic to an augmented Rado graph.

\end{document}